\documentclass[a4paper]{article}
\usepackage[english]{babel}
\usepackage{amssymb,amsmath,amsthm}
\usepackage{xypic}

\input xy

\newtheorem{theorem}{Theorem}
\newtheorem{lemma}{Lemma}
\def\be{\begin{equation}}
\def\ee{\end{equation}}
\def\bal{\begin{aligned}}
\def\eal{\end{aligned}}
\def\bpm{\begin{pmatrix}}
\def\epm{\end{pmatrix}}
\def\lb{\label}
\def\={\;=\;}
\def\+{\,+\,}
\def\-{\,-\,}

\def\to{\longrightarrow}

\def\Co{\mathbb C}
\def\Q{\mathbb Q}
\def\Z{\mathbb Z}

\def\P{\mathbb P}
\def\CC{\mathcal C}
\def\CH{\rm CH}
\def\GL{{\rm GL}}
\def\SL{{\rm SL}}

\def\codim{\rm codim}

\title{Lines crossing a tetrahedron and the Bloch group}
\author{Kevin Hutchinson and  Masha Vlasenko}

\begin{document}
\maketitle

According to B.~Totaro (\cite{Totaro92}), there is a hope that the Chow groups of a field $k$ can be computed using a very small class 
of affine algebraic 
varieties (linear spaces in the right coordinates), whereas the current definition uses essentially all algebraic cycles in affine 
space. In this note we consider a simple modification of  $\CH^2(\rm{Spec} (\emph{k}),3)$ using only linear 
subvarieties in affine spaces and show that it maps  surjectively to the Bloch group $B(k)$ for any infinite field $k$. 
We also describe the kernel of this map.

The second autor is grateful to Anton Mellit, who taught her the idea of passing from linear subspaces to configurations 
(Lemma~1 below) and pointed out the K-theoretical meaning of Menelaus' theorem, and to the organizers of IMPANGA summer school on algebraic geometry for their incredible hospitality and friendly atmosphere.

\section{Lines crossing a tetrahedron}
Let $k$ be an arbitrary infinite field.  Consider the projective spaces $\P^n(k)$ with fixed sets of homogenous 
coordinates $(t_0:t_1:\dots:t_n) \in \P^n(k)$. We call a subspace $L \subset \P^n(k)$ of codimension $r$ \emph{admissible} if
\[
\codim \Bigl(L \cap \{t_{i_1}=\dots=t_{i_s}=0\} \Bigr) \= r+s 
\]    
for every $s$ and distinct $i_1,\dots,i_s$. (Here $\codim (X) > n$ means $X=\emptyset$.) Let 
\[
\CC^r_n \= \Z \Bigl[ \text{admissible}\; L \subset \P^n(k)\,,\; \codim(L) = r \Bigr]
\]
be the free abelian group generated by all admissible subspaces of $\P^n(k)$ of codimension $r$. Then for every $r$ we have a complex
\[
\dots \overset{d}\to \CC^r_{r+2} \overset{d}\to \CC^r_{r+1} \overset{d}\to \CC^r_r \to 0 \to \dots 
\]
(we assume that $\CC^r_n=0$ when $n < r$) with the differential
\be\label{4}
d[L]=\sum (-1)^i [L \cap \{t_i=0\}]
\ee
where every $\{t_i=0\} \subset \P^n(k)$ is naturally identified with $\P^{n-1}(k)$ by throwing away the coordinate $t_i$. We are interested in the homology groups of these complexes $H^r_n = H_n(\CC^r_{\bullet})$. 

For example, one can easily see that $H^1_1 \cong k^*$. Indeed, a hyperplane $\{\sum \alpha_i t_i = 0\}$ is admissible whenever all the coefficients $\alpha_i$ are nonzero, and if we identify
\be\label{1}\bal
& \CC^1_1 \cong \Z[k^*] \qquad \qquad [\{\alpha_0 t_0 + \alpha_1 t_1 = 0\}] \longmapsto \Bigl[\frac{\alpha_1}{\alpha_0}\Bigr]\\
& \CC^1_2 \cong \Z[k^*\times k^*]\qquad [\{ \alpha_0 t_0 + \alpha_1 t_1 + \alpha_2 t_2 = 0 \}] \longmapsto \Bigl[\bigl(\frac{\alpha_1}{\alpha_0},\frac{\alpha_2}{\alpha_1}\bigr)\Bigr]\\
\eal\ee
then the differential $d:\CC^1_2 \to \CC^1_1$ turns into
\[
[(x,y)] \longmapsto [x] - [xy] + [y]\,.
\]
(one can recognize Menelaus' theorem from plane geometry behind this simple computation). Hence we have 
\[
H^1_1 \cong \Z[k^*]\Big/\bigl\{[x] - [xy] + [y] \,:\, x,y \in k^*\bigr\} \cong k^*\,.
\]
Continuing the identifications of ~\eqref{1}, 
$C^1_{\bullet}$ turns into the bar complex for the group $k^*$ (with the term of degree 0 thrown away) and therefore 
\[
H^1_n  = H_n(k^*, \Z)\,,\qquad n \ge 1\,.
\]

Now we switch to $r=2$ and try to compute $H^2_3$. The four hyperplanes $\{t_i=0\}$ form a tetrahedron $\Delta$ in the 3-dimensional projective space $\P^3(k)$ and the line $\ell$ is  admissible if it  
\begin{itemize}
\item[1)] intersects every face of $\Delta$ transversely, i.e. at one point $P_i = \ell \cap \{t_i=0\}$;
\item[2)] doesn't intersect edges $\{t_{i_1}=t_{i_2}=0\}$ of $\Delta$, i.e. all four points $P_0, \dots, P_3 \in \ell$ are different .
\end{itemize}
Therefore it is natural to associate with $\ell$ a number, the cross-ratio of the four points $P_0, \dots, P_3$ on $\ell$. Namely, there is a unique way to identify $\ell$ with $\P^1(k)$ so that $P_0$, $P_1$ and $P_2$ become $0$, $\infty$ and $1$ respectively, and we denote the image of $P_3$ by $\lambda(\ell) \,\in\, \P^1(k) \smallsetminus  \{0,\infty,1\} \= k^* \smallsetminus  \{1\}$. We extend $\lambda$ linearly to a map
\[\bal
\CC^2_3 &\overset{\lambda}\to\; \Z[k^{*}\smallsetminus \{1\}]\\
\sum n_i [\ell_i] &\longmapsto \sum n_i [\lambda(\ell_i)]
\eal\]  

\begin{theorem} Let $\sigma: k^*\otimes k^* \to k^*\otimes k^*$ be the involution $\sigma(x\otimes y)=- y \otimes x$. 
\begin{itemize}
\item[(i)] If $d(\sum n_i [\ell_i])=0$ then $\sum n_i \lambda(\ell_i)\otimes (1-\lambda(\ell_i))=0$ in $(k^*\otimes k^*)_{\sigma}$.
\item[(ii)] Let $L \subset \P^4(k)$ be an admissible plane and $\ell_i = L \cap \{t_i=0\}$, $i=0,\dots,4$. If we denote $x=\lambda(\ell_0)$ and $y=\lambda(\ell_1)$ then
\[
\lambda(\ell_2)\=\frac y x\,,\quad \lambda(\ell_3)\=\frac {1-x^{-1}}{1-y^{-1}}\;\text{ and }\;\lambda(\ell_4)\=\frac {1-x}{1-y}\,.
\]
\item[(iii)] The map induced by $\lambda$ on homology 
\be\label{7}
{\lambda_{*}} : H^2_3 \to B(k)
\ee
is surjective, where 
\[
B(k) \= \frac{{\rm Ker}\Bigl(\;\bal\Z[k^{*}&\smallsetminus\{1\}] \to (k^*\otimes k^*)_{\sigma}\\ [a] &\longmapsto a \otimes (1-a)\eal\;\Bigr)}
{\Big\langle[x]-[y]+\Bigl[\frac y x \Bigr]-\Bigl[\frac {1-x^{-1}}{1-y^{-1}} \Bigr]+\Bigl[\frac {1-x}{1-y} \Bigr]\,,\,x \ne y\Big\rangle} 
\]
is the Bloch group of $k$ (\cite{Suslin90}).

\item[(iv)] We have $H^2_3 \cong H_3(\GL_2(k))/H_3(k^*)$ and the kernel of~\eqref{7}
\[
K = {\rm Ker}\bigl(H^2_3 \overset{\lambda_*} \to B(k)\bigr)
\]
fits into the exact sequence
\be\label{10}
0 \to {\rm Tor}(k^*,k^*)^{\sim} \to K/T(k) \to k^* \otimes K_2(k) \to K_3^M(k)/2 \to 0\,,
\ee
where ${\rm Tor}(k^*,k^*)^{\sim}$ is the unique nontrivial extension of ${\rm Tor}(k^*,k^*)$ by $\Z/2$, and $T(k)$ is a $2$-torsion 
abelian group (conjectured to be trivial). 
\end{itemize}
\end{theorem}

We remark that ${\rm Tor}(k^*,k^*)={\rm Tor}(\mu(k),\mu(k))$ is a finite abelian group if $k$ is a finitely-generated field. 
Furthermore, it is proved in~\cite{Suslin90} that $B(k)$ has the following relation to $K_3(k)$: 
let $K_3^\mathrm{\small ind}(k)$ be the cokernel of the map from Milnor's K-theory $K^M_3(k) \to K_3(k)$, then there is an exact sequence
\be\label{9}
0 \to {\rm Tor}(k^*,k^*)^{\sim} \to K_3^\mathrm{\small ind}(k) \to B(k) \to 0
\ee
In particular, if $k$ is a number field then as a consequence of~\eqref{9} and Borel's theorem (\cite{Borel79}) we have
\[
\dim B(k)\otimes \Q \= r_2\,,
\]
where $r_2$ is the number of pairs of complex conjugate embeddings of $k$ into $\Co$.

\begin{proof}[Proof of~(i) and~(ii).]
One can check that the diagram
\[\xymatrix{
\CC^2_3 
\ar[d]^{\lambda}\ar[rr]^{d}&&
\CC^2_2\ar[d]^{[t_0:t_1:t_2]\longmapsto t_0\otimes (-t_1) \+ (-t_1)\otimes t_2 \+ t_2\otimes t_0 \+ t_0\otimes t_0}\\
\Z[k^{*}\smallsetminus\{1\}]\ar[rr]^{\quad [a]\longmapsto a\otimes(1-a) \quad}&&
 (k^*\otimes k^*)_{\sigma}}
\]
is commutative, and therefore~(i) follows. It is another tedious computation to check~(ii). 
\end{proof}

In the next section we will prove the remaining claims~(iii) and~(iv) and also show that 
\be\label{6}
H^2_n \cong H_n(\GL_2(k),\Z) / H_n(k^*,\Z) \qquad n \ge 3\,.
\ee

\section{Complexes of configurations}

We say that $n+1$ vectors $v_0,\dots,v_n \in k^r$ are \emph{in general position} if every $\le r$ of them are linearly independent. Let $C(r,n)$ be the free abelian group generated by $(n+1)$-tuples of vectors in $k^r$ in general position. For fixed $r$ we have a complex 
\[
\dots  \overset{d}\to  C(r,2)  \overset{d}\to C(r,1)  \overset{d}\to C(r,0) 
\]
with the differential
\be\label{2}
d[v_0,\dots,v_n]=\sum (-1)^i [v_0,\dots,\check{v}_i,\dots,v_n]
\ee
The augmented complex $C(r,\bullet) \to \Z \to 0$ is acyclic. Indeed, if 
\[d \Bigl(\sum n_i [v^i_0,\dots,v^i_n] \Bigr) = 0
\] and $v \in k^r$ is such that all $(n+2)$-tuples $[v,v^i_0,\dots,v^i_n]$ are in general position (such vectors $v$ exist since $k$ is infinite) then
\[
\sum n_i [v^i_0,\dots,v^i_n] \= d \Bigl(\sum n_i [v,v^i_0,\dots,v^i_n] \Bigr)\,.
\]

\begin{lemma}\label{3} $\CC^r_n \cong C(r,n)_{\GL_r(k)}$ for the diagonal action of $\GL_r(k)$ on tuples of vectors. Moreover,  the complex $\CC^r_{\bullet}$ is isomorphic to the truncated complex $C(r,\bullet)_{\GL_r(k),\bullet \ge r}$.
\end{lemma}

\begin{proof}
For $n\ge r$ there is a bijective correspondence between subspaces of codimension $r$ in $\P^n(k)$ and $\GL_r(k)$-orbits on $(n+1)$-tuples $[v_0,\dots,v_n]$ of vectors in $k^r$ satisfying the condition that $v_i$ span $k^r$. It is given by
\[\bal
L \subset \P^n \;&\longmapsto\; [v_0,\dots,v_n]\,,\, v_i = \text{image of $e_i$ in } {k^{n+1}}/\widetilde{L} \cong k^r\\
[v_0,\dots,v_n]  \;&\longmapsto\; \widetilde{L} = {\rm Ker}[v_0,\dots,v_n]^T \subset k^{n+1}
\eal\]
where $\widetilde{L}$ is the unique lift of $L$ to a linear subspace in $k^{n+1}$ and $e_0,\dots,e_n$ is a standard basis in $k^{n+1}$. 

An admissible point in $\P^r(k)$ is a point which doesn't belong to any of the $r+1$ hyperplanes $\{t_i = 0\}$, and for the corresponding vectors $[v_0,\dots,v_r]$ it means that every $r$ of them are linearly independent. For $n>r$ a subspace $L$ of codimension $r$ in $\P^n(k)$ is admissible whenever all the intersections $L \cap \{t_i=0\}$ are admissible in $\P^{n-1}(k)$. Hence it follows by induction that admissible subspaces correspond exactly to $\GL_r(k)$-orbits of tuples ``in general position''. Obviously, differential~\eqref{4} is precisely~\eqref{2} for tuples.
\end{proof}

The tuples of vectors in general position in $k^r$ modulo the diagonal action of $\GL_r(k)$ are called \emph{configurations}, so $C(r,n)_{\GL_r(k)}$  is the free abelian group generated by configurations of $n+1$ vectors in $k^r$.

\begin{proof}[Proof of~(iii) in Theorem 1.] For brevity we denote $C(2,n)$ by $C_n$ and $\GL_2(k)$ by $G$. Since the complex of $G$-modules $C_{\bullet}$ is quasi-isomorphic to $\Z$ we have the hypercohomology spectral sequence with $E^1_{pq}=H_q(G,C_p) \Rightarrow H_{p+q}(G,\Z)$. Since all modules $C_p$ with $p>0$ are free we have $E^1_{pq}=0$ for $p,q>0$ and $E^1_{p0}=(C_p)_G$. If $G_1 \subset G$ is the stabilizer of $\binom10$ then $E^1_{0q} = H_q(G,\Z[G/G_1]) = H_q(G_1,\Z)$ by Shapiro's lemma. We have $k^* \subset G_1$ and $H_q(k^*,\Z)=H_q(G_1,\Z)$ (see Section~1 in~\cite{Suslin}), so $E^1_{0q}=H_q(k^*,\Z)$. Further, $E^2_{p0}=H_p((C_{\bullet})_G)$ and $E^2_{0q}=H_q(k^*,\Z)$. This spectral sequence degenerates on the second term. Indeed, the embedding 
\[\bal
k^* &\hookrightarrow G \\
\alpha &\mapsto \bpm1&0\\0&\alpha\epm
\eal\] 
is split by determinant, and therefore all maps $H_q(k^*,\Z)\to H_q(G,\Z)$ are injective. Consequently, $E^{\infty}_{pq}=E^2_{pq}$ and for every $n \ge 2$ we have a short exact sequence
\[
0 \to H_n(k^*,\Z)\to H_n(G,\Z) \to H_n\bigl((C_{\bullet})_G \bigr)\to 0\,.
\]
It follows from Lemma~\ref{3} that
\[
H^2_n \=  H_n\bigl((C_{\bullet})_G\bigr) \= H_n(G,\Z) / H_n(k^*,\Z)\,, \qquad n \ge 3\,.
\]

Let $D_n$ be the free abelian group generated by  $(n+1)$-tuples of distinct points in $\P^1(k)$. Again we have the differential like~\eqref{2} on $D_{\bullet}$ and the augmented complex $D_{\bullet}\to\Z\to 0$ is acyclic. We have a surjective map from $C_{\bullet}$ to $D_{\bullet}$ since a non-zero vector in $k^2$ defines a point in $P^1(k)$ and the group action agrees. The spectral sequence $\widetilde E^1_{pq}=H_q(G,D_p) \Rightarrow H_{p+q}(G,\Z)$ was considered in~\cite{Suslin90}.  In particular, $\widetilde E^1_{p0} \= (D_p)_G$ is the free abelian group generated by $(p-2)$-tuples of different points since $G$-orbit of every $(p+1)$-tuple contains a unique element of the form $(0,\infty,1,x_1,\dots,x_{p-2})$, and the differential $d^1: \widetilde E^1_{04} \to \widetilde E^1_{03}$ is given by
\be\label{5}
 [x,y] \mapsto [x]-[y]+\Bigl[\frac y x \Bigr]-\Bigl[\frac {1-x^{-1}}{1-y^{-1}} \Bigr]+\Bigl[\frac {1-x}{1-y} \Bigr]\,.
\ee
According to~\cite{Suslin90}, terms $\widetilde E^2_{pq}$ with small indices are
\[\xymatrix{
H_3(k^*\oplus k^*)&&&\\
H_2(k^*)\oplus(k^*\otimes k^*)_{\sigma} &(k^*\otimes k^*)^{\sigma}&&\\
k^*&0&0&\\
\Z&0&0&\mathfrak{p}(k)}\]
where $\mathfrak{p}(k)$ is the quotient of $\Z[k^* \smallsetminus \{ 1 \}]$ by all 5-term relations as in right-hand side of~\eqref{5}, and the only non-trivial differential starting from $\mathfrak{p}(k)$ is 
\[\bal
d^3:\mathfrak{p}(k) &\to H_2(k^*)\oplus(k^*\otimes k^*)_{\sigma} =\Lambda^2 (k^*)\oplus(k^*\otimes k^*)_{\sigma}\\
[x]\;&\mapsto \; x \wedge (1-x) \- x \otimes(1-x)
\eal\]
Therefore $\widetilde E^4_{30}=\widetilde E^{\infty}_{30}=B(k)$ and we have a commutative triangle
\[\xymatrix{
H_3(G) \ar@{>>}[dr] \ar@{>>}[r] & E^{\infty}_{30}=H^2_3 \ar[d]\\
& \; \widetilde E^{\infty}_{30} = B(k)
}\]
where both maps from $H_3(G)$ are surjective, hence the vertical arrow is also surjective. It remains to check that the vertical arrow coincides with $\lambda_*$. A line $\ell$ in $\P^3(k)$ is given by two linear equations and for an admissible line it is always possible to chose them in the form
\[\begin{cases}
t_0 \phantom{+ t_1} + x_1 t_2 + x_2 t_3 \= 0\,,\\
\phantom{t_0 +}  t_1 + y_1 t_2 + y_2 t_3 \= 0\,.
\end{cases}\]
This line corresponds to the tuple of vectors
\[
\binom 10\,,\; \binom 01\,,\; \binom{x_1}{y_1}\,,\; \binom{x_2}{y_2} 
\]
which can be mapped to the points $0,\infty,1, \frac{x_1 y_2}{y_1 x_2}$ in $\P^1(k)$, hence the vertical arrow maps it to $[\frac{x_1 y_2}{y_1 x_2}]$ (actually we need to consider a linear combination of lines which vanishes under $d$ but for every line the result is given by this expression). On the other hand, four points of its intersection with the hyperplanes are
\[\bal
&P_0 \= (0: y_1x_2-y_2x_1:-x_2:x_1)\\
&P_1 \= (y_2x_1-y_1x_2:0:-y_2:y_1)\\
&P_2 \= (-x_2:-y_2:0:1)\\
&P_3 \= (-x_1:-y_1:1:0)\\
\eal\]
and if we represent every point on $\ell$ as $\alpha P_0 + \beta P_1$ then the corresponding ratios $\frac{\beta}{\alpha}$ will be $0,\infty,-\frac{x_2}{y_2}, -\frac{x_1}{y_1}$.  Hence $\lambda(\ell)=\frac{x_1 y_2}{y_1 x_2}$ again and~(iii) follows.

To prove~(iv) we first observe that the Hochschild-Serre spectral sequence associated to
\[
1 \to \SL_2(k) \to \GL_2(k) \overset{\det}\to k^* \to 1
\]
gives a short exact sequence
\be\lb{8}\bal
1 \to H_0\Bigl(k^*,H_3(\SL_2(k),\Z) \Bigr) \to {\rm Ker}&\Bigl(H_3(\GL_2(k),\Z)\overset{\det}\to H_3(k^*,\Z)\Bigr) \\
&\to H_1\Bigl(k^*, H_2(\SL_2(k),\Z) \Bigr)\to 1.
\eal\ee

The first term here maps surjectively to $K_3^\mathrm{\small ind}(k)$ (see the last section of \cite{HT}), 
and the map is conjectured by Suslin to be an isomorphism (see Sah \cite{sah:discrete3}). 
It is known that its kernel is at worst 2-torsion (see Mirzaii \cite{mirzaii:third}).

Thus we let 
\[
T(k):={\rm Ker}\left(H_0(k^*,H_3(\SL_2(k),\Z) )\to K_3^\mathrm{\small ind}(k)\right).
\]
By the preceeding remarks, this is a $2$-torsion abelian group. 
Since the embedding $k^* \to \GL_2(k)$ is split by the determinant, the middle term in~\eqref{8} is isomorphic to $H^2_3$. 
Then applying the snake lemma to 
the diagram
\begin{eqnarray*}
\xymatrix{
0\ar[r]
&T(k)\ar[r]\ar@{^(->}[d]
&H_0\Bigl(k^*,H_3(\SL_2(k),\Z) \Bigr)\ar[r]\ar@{^(->}[d]
&K_3^\mathrm{\small ind}(k)\ar[r]\ar@{>>}[d]
&0\\
0\ar[r]
&
K\ar[r]
&
H^3_2\ar[r]
&
B(k)\ar[r]
&
0}
\end{eqnarray*}
gives the short exact sequence
\[
0\to{\rm Tor}(k^*,k^*)^{\sim} \to  K/T(k)  \to H_1\Bigl(k^*,H_2(\SL_2(k),\Z)\Bigr) \to 0.
\]
Finally, it  follows from~\cite{HT} that there is a natural short exact sequence
\[
0 \to H_1\Bigl(k^*,H_2(\SL_2(k),\Z)\Bigr) \to k^* \otimes K_2^M(k) \to K_3^M(k)/2 \to 0\,.
\]
This proves~\eqref{10}.
\end{proof}

\end{document}